\documentclass{amsart}
\usepackage{graphicx}
\usepackage{setspace}
\usepackage[OT2,T1]{fontenc}
\DeclareSymbolFont{cyrletters}{OT2}{wncyr}{m}{n}
\DeclareMathSymbol{\Sha}{\mathalpha}{cyrletters}{"58}
\newtheorem{thm}{Theorem}[section]
\newtheorem{cor}[thm]{Corollary}
\newtheorem{lem}[thm]{Lemma}

\theoremstyle{definition}
\newtheorem{defn}[thm]{Definition}
\theoremstyle{remark}
\newtheorem{rem}[thm]{Remark}
\numberwithin{equation}{section}

\begin{document}

\title[Matrix Weighted Martingale Transforms]{Martingale transforms, the dyadic shift and the Hilbert transform: a
sufficient condition for boundedness between matrix weighted spaces}%
\author{Robert Kerr}%
\address{Department of Mathematics, University Gardens, University of Glasgow, G12 8QW}%
\email{rkerr@maths.gla.ac.uk}%

\thanks{This work was completed with the support of the EPSRC. Part of this work was done while visiting the Fields Institute as part of the Thematic Program on New Trends in Harmonic Analysis}%

\begin{abstract}
We give sufficient conditions on $N\times N$ matrix weights $U$ and
$V$ for the dyadic martingale transforms to be uniformly bounded
from $L^2(V)$ to $L^2(U)$. We also show that these conditions imply
the uniform boundedness of the dyadic shifts as well as the
boundedness of the Hilbert transform between these spaces.

\end{abstract}
\maketitle

\section{Introdution}

Much progress has been made recently on the two-weight problem for
various important operators, for example the Sawyer type
characterizations of F. Nazarov, S. Treil and A. Volberg, see e.g.
\cite{band}, and the two-weight inequalities for maximal singular
integrals by M. Lacey, E.T. Sawyer and I. Uriarte-Tuero \cite{ls}.
This is currently an area of much activity and new proofs with
broader scope and deeper insight are appearing. Little attention has
been given so far to understanding two-weight problems on
vector-valued function spaces (the work of C. M. Pereyra and N. H.
Katz \cite{kp} being a notable exception), in contrast to the
one-weight case, for which an analogue of the
Hunt-Muckenhoupt-Wheeden characterization has been shown in
\cite{wavelet} by S. Treil and A. Volberg. A sufficient condition
for the operator weight case has been given by S. Pott in
\cite{pott} and in \cite{gptv} it is shown that the dyadic operator
weight analogue of the matrix weight dyadic Hunt-Muckenhoupt-Wheeden
condition does not imply the boundedness of the martingale
transforms. We turn our attention firstly to conditions which imply
the uniform boundedness of dyadic martingale transforms and then to
other dyadic operators. The motivation here is that such dyadic
operators can often be used as models for more general singular
integral operators.

\section{The Martingale Transform}

 Let $\mathcal{D}$ denote the standard grid of dyadic subintervals of $\mathbb{R}$, $\mathcal{D} =  \left[k2^{-n},(k+1)2^{-n}\right)$ where
$n$ and $k$ range over the integers. The Haar functions associated
to a dyadic interval $I$ are defined as $h_I =
\frac{1}{\sqrt{I}}\left(\chi_{I_-} - \chi_{I_+}\right),$ where $I_-$
and $I_+$ are the largest proper dyadic subintervals of $I$, on the
right and the left respectively. The $h_I$ form an orthonormal basis
for $L^2(\mathbb{R})$. Let $L^2(\mathbb{R},\mathbb{C}^n)$ denote the
space of measurable functions
\[ \left\{ f:\mathbb{R} \rightarrow \mathbb{C}^n:\int_{\mathbb{R}} \langle f(t) ,
f(t) \rangle_{\mathbb{C}^n}dt < \infty\right\},\]
\[||f||_{L^2(\mathbb{R},\mathbb{C}^n)} = \left(\int_{\mathbb{R}}
\langle f(t) , f(t) \rangle_{\mathbb{C}^n}dt\right)^{\frac{1}{2}}.\]
We consider the operator $T_\sigma$ on
$L^2(\mathbb{R},\mathbb{C}^n)$ defined by the mapping
\[T_\sigma f \mapsto \sum_{I\in\mathcal{D}} \sigma(I)h_If_I\] where
$f_I = \int_I f h_I$ is the Haar coefficient for $I$ and $\sigma(I)
= \pm1$. The $T_\sigma$ are \emph{dyadic martingale transforms} and
are unitary operators on $L^2(\mathbb{R},\mathbb{C}^n)$. For a
matrix valued function $V$ which is positive and invertible almost
everywhere, let $L^2 (\mathbb{R},\mathbb{C}^n,V)$ be the space of
measurable functions
\[ \left\{f:\mathbb{R} \rightarrow \mathbb{C}^n:\int_{\mathbb{R}} \langle V(t)^{\frac{1}{2}}f(t),V(t)^{\frac{1}{2}}f(t)
\rangle dt < \infty\right\}\] with norm
\[||f||_{L^2(\mathbb{R},\mathbb{C}^n,V)} = \left(\int_{\mathbb{R}}\langle
V(t)^{\frac{1}{2}}f(t),V(t)^{\frac{1}{2}}f(t) \rangle
dt\right)^{\frac{1}{2}}.\] This generalizes the notion of weighted
$L^2$ spaces of scalar functions where a weight is a measurable
almost everywhere positive  function. We refer to matrix functions
which are measurable, almost everywhere positive and invertible as
\emph{matrix weights}. The purpose of this paper is to find
conditions on a pair of matrix weights, $U$ and $V$, which imply
that the dyadic martingale transforms are uniformly bounded from
$L^2 (\mathbb{R},\mathbb{C}^n,V)$ to
$L^2(\mathbb{R},\mathbb{C}^n,U)$.  This is equivalent to showing
that the operators $M_V^{-\frac{1}{2}}T_\sigma M_U^{\frac{1}{2}}$
are uniformly bounded on the unweighted space
$L^2(\mathbb{R},\mathbb{C}^n)$. The sufficient conditions we find on
a pair of matrix weights are a joint $A_2$ condition, a matrix
$A_\infty$ condition on one weight and a matrix reverse H\"{o}lder
condition on the other weight. We can also as a corollary replace
the matrix reverse H\"{o}lder condition by the matrix $A_\infty$
condition. The matrix $A_\infty$ and reverse H\"{o}lder condition
will be discussed in the next section. In what follows we will
denote $L^2 (\mathbb{R},\mathbb{C}^n,V)$ and $L^2
(\mathbb{R},\mathbb{C}^n,U)$ by $L^2(V)$ and $L^2(U)$.

\section{The $A_{2,0}$ condition and reverse H\"{o}lder}
\begin{defn}\label{revh}

A matrix weight $U$ satisfies the \emph{dyadic reverse H\"{o}lder
inequality} if there exists constants $C > 0$ and $r > 2$ such that
\[\int_I
||U^{\frac{1}{2}}(x)\left<U\right>_I^{-\frac{1}{2}}y||^rdx \leq
C|I|||y||^r
\]
holds for all dyadic intervals $I$ and nonzero vectors $y.$

\end{defn}

Note that our definition of the reverse H\"{o}lder property is in
general weaker the existing definition in the literature by Christ
and Goldberg \cite{goldberg}, but is equivalent for finite
dimensional spaces. Our definition generalizes the scalar version
and is in a form we find applicable.

\begin{defn}

We say that a matrix weight $U$ is in the $A_{2,0}$ class of weights
if following inequality holds uniformly over all intervals $I$;
\[ \det\langle U\rangle_I \leq C \exp\{\langle \log\det U
\rangle_I\}.\]

\end{defn}
This $A_{2,0}$ condition is a matrix analog of the scalar $A_\infty$
condition, see \cite{naztreil} for discussion on this. Also see
\cite{marc} for some reformulations and context of this property.

\begin{lem}\label{aor}

 If a matrix weight $U$ satisfies the $A_{2,0}$ condition, then it satisfies the
reverse H\"{o}lder inequality.
\end{lem}
\begin{proof}
 If the weight $U$ has the $A_{2,0}$ condition, then by Lemma 3.2 and Lemma 3.3 of \cite{marc} we have that

\begin{equation}\label{uniformityx} \left\{\frac{1}{|I|}\int_I
||U^{\frac{1}{2}}x||^2\right\}^{\frac{1}{2}} \leq C \exp \{ \langle
\log ||U^{\frac{1}{2}}x|| \rangle_I\}\end{equation}
 for all nonzero
$x$ and intervals $I$. Consequently,
\[ \frac{1}{|I|}\int_I ||U^{\frac{1}{2}}x||^2 \leq C \exp \{ \langle \log ||U^{\frac{1}{2}}x||^2 \rangle_I\}\]
and thus the scalar weight $||U^{\frac{1}{2}}x||^2$ satisfies the
$A_\infty$ condition and hence a reverse H\"{o}lder inequality;
\[ \left\{\frac{1}{|I|}\int_I ||U^{\frac{1}{2}}x||^{2r}\right\}^{\frac{1}{2r}} \leq C \left\{\frac{1}{|I|}\int_I ||U^{\frac{1}{2}}x||^2\right\}^{\frac{1}{2}}\] for some $r > 1$, all intervals $I$ and all nonzero $x$.
Note that the index $r$ does not depend on $x$ because it only
depends on the $A_\infty$ constant $C$ in (\ref{uniformityx}), which
is uniform for all $x.$ As this is true for all nonzero $x$, we can
replace $x$ by $\langle U \rangle_I^{-\frac{1}{2}} y$, where $0 \neq
y \in \mathbb{C}^n$. Thus for all intervals $I \in \mathbb{R}$ and
$y \in \mathbb{C}^n$
\[\left\{\frac{1}{|I|}\int_I ||U^{\frac{1}{2}}\langle U \rangle_I^{-\frac{1}{2}}y||^{2r}\right\}^{\frac{1}{2r}} \leq C \left\{\frac{1}{|I|}\int_I ||U^{\frac{1}{2}}\langle U\rangle_I^{-\frac{1}{2}}y||^2\right\}^{\frac{1}{2}} = C||y||.\]
\end{proof}

\section{Boundedness of the martingale transform}

We are now in a position to state our main theorem concerning
sufficient conditions for the boundedness of the dyadic martingale
transforms:

\begin{thm}\label{maint2}
Let $U$ and $V$ be matrix weights satisfying the joint $A_2$
condition
\[ \left<V^{-1}\right>^{\frac{1}{2}}_I\left<U\right>_I\left<V^{-1}\right>^{\frac{1}{2}}_I < C\] for all dyadic intervals $I$, where $C$ is a constant multiple of the identity. If $V^{-1} \in A_{2,0}$ and $U$ satisfies the matrix reverse H\"{o}lder inequality, then the dyadic martingale transforms are uniformly
bounded from $L^2(V)$ to $L^2(U).$

\end{thm}

\begin{cor}\label{maint}
Let $U$ and $V$ be matrix weights satisfying a joint $A_2$ condition
\[ \left<V^{-1}\right>^{\frac{1}{2}}_I\left<U\right>_I\left<V^{-1}\right>^{\frac{1}{2}}_I < C\] for all dyadic intervals $I$, where $C$ is a constant multiple of the identity. If $U$ and
$V^{-1}$ are also in $A_{2,0}$, then the dyadic martingale
transforms are uniformly bounded from $L^2(V)$ to $L^2(U).$

\end{cor}
\begin{proof}
By Lemma \ref{aor}, Theorem \ref{maint2} implies this corollary.

\end{proof}

Note that the conditions on the matrix weights $U$ and $V^{-1}$ are
symmetric in this corollary. Also in Theorem 6.1 of \cite{lacey} the
conditions and implications in Corollary \ref{maint} are stated but
specifically for the scalar valued function space setting, this is
also mentioned in \cite{bellman}.

\section{Proof of Theorem \ref{maint2} using a two-weighted dyadic square function}
We introduce the operator $D_{V^{-1}}$ defined by
\[D_{V^{-1}}f = D_{V^{-1}}\sum_{I \in \mathcal{D}} f_I
h_I(x) \mapsto \sum_{I \in \mathcal{D}}
\left<V^{-1}\right>^{\frac{1}{2}}_I f_I h_I(x)\] for functions with
finite Haar expansion.

Write $M_V^{-\frac{1}{2}}T_\sigma M_U^{\frac{1}{2}}$ as
$M_V^{-\frac{1}{2}}T_\sigma D^{-1}_{V^{-1}}
D_{V^{-1}}M_U^{\frac{1}{2}}$ and note that $T_\sigma$ and
$D^{-1}_{V^{-1}}$ commute. This allows us to estimate the norm as
\[||M_V^{-\frac{1}{2}}T_\sigma M_U^{\frac{1}{2}}|| = ||M_V^{-\frac{1}{2}}T_\sigma D^{-1}_{V^{-1}} D_{V^{-1}}M_U^{\frac{1}{2}}|| \leq ||M_V^{-\frac{1}{2}} D^{-1}_{V^{-1}}|| ||T_\sigma|| ||D_{V^{-1}}M_U^{\frac{1}{2}}||.\]
We know that $T_\sigma$ is bounded on unweighted $L^2$ so we are
interested in finding conditions on the matrix weights $U$ and
$V^{-1}$ that imply the boundedness of the operators
$M_V^{-\frac{1}{2}} D^{-1}_{V^{-1}}$ and
$D_{V^{-1}}M_U^{\frac{1}{2}}$ on unweighted
$L^2(\mathbb{R},\mathbb{C}^n).$

We deal with $D_{V^{-1}}M_U^{\frac{1}{2}}$, a two-weighted dyadic
square function, using a stopping time argument and Cotlar's Lemma.

\begin{thm}\label{square}
Let $U$ and $V^{-1}$ be matrix weights such that $U$ has the
\emph{dyadic reverse H\"{o}lder inequality} and such that for all
dyadic intervals $I,$
\[\left<V^{-1}\right>^{\frac{1}{2}}_I\left<U\right>_I\left<V^{-1}\right>^{\frac{1}{2}}_I < C.\] Then the
two-weighted square function $S = M_U^{\frac{1}{2}}D_{V^{-1}}$ is
bounded on $L^2(\mathbb{R},\mathbb{C}^n).$
\end{thm}
For $M_V^{-\frac{1}{2}} D^{-1}_{V^{-1}}$ we state without proof a
theorem of Nazarov and Treil.
\begin{thm}\label{naztreil}
Let $U$ be a matrix weight such that $W \in A_{2,0}$. Then
$D_{W}^{-1}M^{\frac{1}{2}}_{W}$ is bounded on $L^2(\mathbb{R},
\mathbb{C}^n)$.
\end{thm}
\begin{proof}
This is Theorem 7.8 of \cite{naztreil}. Note that the proof of this
theorem uses a Bellman function technique.
\end{proof}
This theorem also applies to $M_V^{-\frac{1}{2}}D^{-1}_{V^{-1}}$ if
we note that its adjoint is $D^{-1}_{V^{-1}}M_V^{-\frac{1}{2}}$.

We now introduce the stopping time used in the proof of Theorem
\ref{square}.
\subsection{Stopping Time}  Let $\lambda > 1$ and let
$\mathcal{J}_{\lambda,1}(J)$ be the collection of maximal dyadic
subintervals $I_{\lambda}$ of $J$ such that
\begin{equation}\label{stone} ||\frac{1}{|I_{\lambda}|} \int_{I_{\lambda}}
\left<V^{-1}\right>_J^{\frac{1}{2}}U(x)\left<V^{-1}\right>_J^{\frac{1}{2}}dx||
> \lambda
\end{equation}
or
\begin{equation}\label{sttwo}  ||\frac{1}{|I_{\lambda}|} \int_{I_{\lambda}}
\left<V^{-1}\right>_J^{-\frac{1}{2}}V^{-1}(x)\left<V^{-1}\right>_J^{-\frac{1}{2}}dx||
> \lambda
\end{equation}
or
\begin{equation}\label{stthree}  ||\frac{1}{|I_{\lambda}|} \int_{I_{\lambda}}
\left<U\right>_J^{-\frac{1}{2}}U(x)\left<U\right>_J^{-\frac{1}{2}}dx||
> \lambda.
\end{equation}
Then we define $\mathcal{J}_{\lambda,k}(J)$ as $\cup_{I \in
\mathcal{J}_{\lambda,k-1}(J)} \mathcal{J}_{\lambda,1}(I)$ for $k >
1.$ Let $\mathcal{F}_{\lambda,1}(J)$ be the collection of those
dyadic subintervals of $J$ which are not a subinterval of any
interval in $\mathcal{J}_{\lambda,1}(J)$. We likewise define
$\mathcal{F}_{\lambda,k}(J)$ iteratively to be $\cup_{I \in
\mathcal{J}_{\lambda,k-1}(J)} \mathcal{F}_{\lambda,1}(I).$ Then
$\cup_k \mathcal{F}_{\lambda,k}(I)$ forms a decomposition of the
dyadic subintervals of $I.$

\begin{lem}{\label{dec}}

If $U$ and $V$ are matrix weights such that for some $C > 0$
\[ \left<V^{-1}\right>^{\frac{1}{2}}_I\left<U\right>_I\left<V^{-1}\right>^{\frac{1}{2}}_I < C\] for all dyadic intervals $I$, then $\mathcal{J}$ is a decaying
stopping time for some $\lambda>1$. By decaying stopping time, we
mean that for a sufficiently large $\lambda$, we have a constant
$0<\delta<1$ such that $|\mathcal{J}(I)_{\lambda,k}| \leq
\delta^k|I|$ for all $k$. We call this conditions on the two matrix
weights the joint $A_2$ condition.

\end{lem}

\begin{proof}
We first restrict ourselves to showing that
$\mathcal{J}_{\lambda,k}(J)=\cup_{I \in
\mathcal{J}_{\lambda,k-1}(J)} \mathcal{J}_{\lambda,1}(I)$ is a
decaying stopping time when $\mathcal{J}_{\lambda,1}(I)$ is defined
as the collection of maximal subintervals of $I$ satisfying only
(\ref{stthree}) rather than all three conditions.

We have the following series of inequalities;
\[ |I| \geq ||\int_I \langle U \rangle_I^{-\frac{1}{2}} U(x)  \langle U
\rangle_I^{-\frac{1}{2}} dx|| \geq ||\sum_{J
\in\mathcal{J}_{\lambda,1}} \int_{J}\langle U
\rangle_I^{-\frac{1}{2}} U(x) \langle U \rangle_I^{-\frac{1}{2}}dx||
\]\[\geq C_n\sum_{J \in \mathcal{J}_{\lambda,1}}
||\int_{J}\langle U \rangle_I^{-\frac{1}{2}} U(x) \langle U
\rangle_I^{-\frac{1}{2}}||,\]

where $C_n$ is a constant dependent on the matrix. This is possible
due to the equivalence of all matrix norms and the additivity of the
trace norm on positive matrices. By (\ref{stthree}),
\[C_n\sum_{J\in\mathcal{J}_{\lambda,1}}
||\int_{J}\langle U \rangle_I^{-\frac{1}{2}} U(x) \langle U
\rangle_I^{-\frac{1}{2}}|| \geq
C_n\lambda\sum_{J\in\mathcal{J}_{\lambda,1}}|J|\] and hence
\[\frac{1}{\lambda C_n} |I| \geq \sum_{J\in\mathcal{J}_{\lambda,1}}|J|= |\mathcal{J}_{\lambda,1}|.\]

Thus we can choose $\lambda$ to be large enough such that
$\frac{1}{\lambda C_n} < 1$ and we have
$|\mathcal{J}_{\lambda,1}(I)| < \delta|I|$. Iteration now yields
that $|\mathcal{J}_{\lambda,k}(I)| < \delta^k |I|.$ We use a similar
argument for \ref{stone} and \ref{sttwo} individually and then note
that the finite union of decaying stopping times will also be a
decaying stopping time, after a possible change of $\lambda$.

\end{proof}

\subsection{Proof of Theorem \ref{square}}

This proof of this theorem is where the core of our technical
analysis takes place, it draws from Theorem 3.1 in \cite{pott}. We
are presenting a generalization for the finite dimensional case.

\begin{proof}

We choose $\lambda > 0$ such that the condition $\mathcal{J}$ in
\ref{dec} is a decaying stopping time. First note that almost
everywhere on $J \backslash \cup \mathcal{J}(J)$:
\[\left<V^{-1}\right>_J^{\frac{1}{2}}U(x)\left<V^{-1}\right>_J^{\frac{1}{2}}
\leq \lambda\]
\[
\left<U\right>_J^{-\frac{1}{2}}U(x)\left<U\right>_J^{-\frac{1}{2}}
\leq \lambda
\]
and
\[
\left<V^{-1}\right>_J^{-\frac{1}{2}}V^{-1}(x)\left<V^{-1}\right>_J^{-\frac{1}{2}}
\leq \lambda.
\]
In this context $\lambda$ stands for the identity matrix scaled by
$\lambda$, and the inequalities are matrix inequalities. Let us take
$f \in L^2(\mathbb{R}, \mathbb{C}^n)$ with finite Haar expansion.
Assume without loss that $f$ is supported in the unit interval. We
write $\mathcal{J}_j$ and $\mathcal{F}_j$ for
$\mathcal{J}_{\lambda,j}([0,1])$ and
$\mathcal{F}_{\lambda,j}([0,1]).$

Define
\[\triangle_j f = \sum_{K\in\mathcal{F}_j}h_Kf_K\]
and
\[S_jf = S\triangle_j f =
U^{\frac{1}{2}}\sum_{K\in\mathcal{F}_j}\left<V^{-1}\right>^{\frac{1}{2}}_Kh_Kf_K.\]
We can check that $\sum_{j=1}^{\infty} \triangle_jf = f$ and also
that
\[\sum_{j=1}^{\infty} S_jf = Sf.\]
We show that $S$ is bounded using Cotlar's Lemma. First note that
\[||S_jf||_{L^2}^2 = \int_{\cup \mathcal{J}_{j-1}}
||S_jf||_{\mathbb{C}^n}^2dx = \int_{\cup \mathcal{J}_{j-1} \setminus
\cup \mathcal{J}_j} ||S_jf||_{\mathbb{C}^n}^2dx + \int_{\cup
\mathcal{J}_j} ||S_jf||_{\mathbb{C}^n}^2dx.\] We estimate
$\int_{\cup \mathcal{J}_{j-1} \setminus \cup \mathcal{J}_j}
||S_jf||_{\mathbb{C}^n}^2dx$ and then $\int_{\cup \mathcal{J}_j}
||S_jf||_{\mathbb{C}^n}^2dx.$

\begin{align*}
& \int_{\cup \mathcal{J}_{j-1} \setminus \cup \mathcal{J}_j}
||S_jf||_{\mathbb{C}^n}^2dx = \sum_{J \in \mathcal{J}_{j-1}} \int_{J
\setminus \cup \mathcal{J}(J)}||S_jf||_{\mathbb{C}^n}^2dx
\\&=\sum_{J \in \mathcal{J}_{j-1}} \int_{J \setminus \cup
\mathcal{J}(J)}||U^{\frac{1}{2}}(x)\sum_{K\in\mathcal{F}(J)}\left<V^{-1}\right>_K^{\frac{1}{2}}h_K(x)f_K||_{\mathbb{C}^n}^2dx
\\&=\sum_{J \in \mathcal{J}_{j-1}} \int_{J \setminus \cup
\mathcal{J}(J)}||U^{\frac{1}{2}}(x)\left<V^{-1}\right>_J^{\frac{1}{2}}\left<V^{-1}\right>_J^{-\frac{1}{2}}\sum_{K\in\mathcal{F}(J)}\left<V^{-1}\right>_K^{\frac{1}{2}}h_K(x)f_K||_{\mathbb{C}^n}^2dx
\\&\leq\sum_{J \in \mathcal{J}_{j-1}} \int_{J \setminus \cup
\mathcal{J}(J)}||U^{\frac{1}{2}}(x)\left<V^{-1}\right>_J^{\frac{1}{2}}||^2
||\left<V^{-1}\right>_J^{-\frac{1}{2}}\sum_{K\in\mathcal{F}(J)}\left<V^{-1}\right>_K^{\frac{1}{2}}h_K(x)f_K||_{\mathbb{C}^n}^2dx
\\&\leq\sum_{J \in \mathcal{J}_{j-1}} \int_{J \setminus \cup
\mathcal{J}(J)}\lambda
||\left<V^{-1}\right>_J^{-\frac{1}{2}}\sum_{K\in\mathcal{F}(J)}\left<V^{-1}\right>_K^{\frac{1}{2}}h_K(x)f_K||_{\mathbb{C}^n}^2dx
\\&\leq\sum_{J \in \mathcal{J}_{j-1}} \int_{J }\lambda \sum_{K\in\mathcal{F}(J)}
||\left<V^{-1}\right>_J^{-\frac{1}{2}}\left<V^{-1}\right>_K^{\frac{1}{2}}||^2
||f_K||_{\mathbb{C}^n}^2dx
\\&\leq\sum_{J \in \mathcal{J}_{j-1}} \int_{J}\lambda \sum_{K\in\mathcal{F}(J)} \lambda
||f_K||_{\mathbb{C}^n}^2dx
\end{align*}
since for $K \in \mathcal{F}(J),$
\[\left<V^{-1}\right>_J^{-\frac{1}{2}}\left<V^{-1}\right>_K\left<V^{-1}\right>_J^{-\frac{1}{2}}
=
\frac{1}{|K|}\int_K\left<V^{-1}\right>_J^{-\frac{1}{2}}V^{-1}(x)\left<V^{-1}\right>_J^{-\frac{1}{2}}
\leq \lambda.\] Thus
\[||S_jf||_{L^2}^2 \leq \sum_{J \in
\mathcal{J}_{j-1}} \int_{J }\lambda \sum_{K\in\mathcal{F}(J)}
\lambda ||f_K||_{\mathbb{C}^n}^2dx  \]
\[=\sum_{J \in
\mathcal{J}_{j-1}} \lambda^2 \sum_{K\in\mathcal{F}(J)} \int_{J}
||f_K||_{\mathbb{C}^n}^2dx = \lambda^2 ||\triangle_jf||_{L^2}.\]

We now consider
\[\int_{\cup \mathcal{J}_j}
||S_jf||_{\mathbb{C}^n}^2dx = \sum_{J\in \mathcal{J}_{j-1}} \sum_{I
\in \mathcal{J}(J)} \int_I ||U^{\frac{1}{2}}(x)\sum_{K\in
\mathcal{F}(J)}
\left<V^{-1}\right>^{\frac{1}{2}}f_Kh_K(x)||_{\mathbb{C}^n}^2dx\] As
$h_K$ is constant on $I \in \mathcal{J}(J)$ for $K \in
\mathcal{F}(J),$ this is equal to
\[\sum_{J\in \mathcal{J}_{j-1}} \sum_{I
\in \mathcal{J}(J)} \int_I
||\left<U\right>^{\frac{1}{2}}_I\sum_{K\in \mathcal{F}(J)}
\left<V^{-1}\right>_K^{\frac{1}{2}}f_Kh_K||_{\mathbb{C}^n}^2dx
\]
\[\leq\sum_{J\in \mathcal{J}_{j-1}} \sum_{I
\in \mathcal{J}(J)} \int_I
||\left<U\right>^{\frac{1}{2}}_I\left<U\right>^{-\frac{1}{2}}_J||_{\mathbb{C}^n}^2
||\left<U\right>^{\frac{1}{2}}_J\sum_{K\in \mathcal{F}(J)}
\left<V^{-1}\right>_K^{\frac{1}{2}}f_Kh_K||_{\mathbb{C}^n}^2dx\leq
2\lambda^2||\triangle_jf||_{L^2}^2.\]

We have shown that there is a constant $C$ such that $||S_jf||^2
\leq C||\triangle_jf||^2.$ Let us now show that there exists a
constant $C^\prime$ and $0 < d < 1$ such that for $k >j,$
\[\int_{\cup \mathcal{J}_{k-1}} ||S_jf||^2dx \leq C^\prime d^{k-j}
||\triangle_jf||^2.\] Cotlar's Lemma (see \cite{pereyra}) then
implies that $S = \sum S_j$ is bounded. Note that
\[\int_{\cup \mathcal{J}_{k-1}} ||S_jf||^2dx =
\sum_{J\in\mathcal{J}_j}\sum_{I\in\mathcal{J}_{k-j-1}(J)}\int_I
||U^{\frac{1}{2}}(x)\sum_{L\in \mathcal{J}_{j-1}}\sum_{K\in
\mathcal{F}(L)}
\left<V^{-1}\right>^{\frac{1}{2}}_Kh_K(x)f_K||_{\mathbb{C}^n}^2dx\]

Note that $\sum_{L\in \mathcal{J}_{j-1}}\sum_{K\in \mathcal{F}(L)}
\left<V^{-1}\right>^{\frac{1}{2}}_Kh_Kf_K$ is constant on $J \in
\mathcal{J}_j,$ and denote this constant by $M_Jf$. The above
expression is equal to
\begin{align*}
&\sum_{J\in\mathcal{J}_j}\sum_{I\in\mathcal{J}_{k-j-1}(J)}
|I|||\left<U\right>_I^{\frac{1}{2}}M_Jf||_{\mathbb{C}^n}^2
\\&=\sum_{J^{\prime} \in
\mathcal{J}_{j-1}}\sum_{J\in\mathcal{J}(J^\prime)}\sum_{I\in\mathcal{J}_{k-j-1}(J)}|I|||\left<U\right>_I^{\frac{1}{2}}M_Jf||_{\mathbb{C}^n}^2
\\&=\sum_{J^{\prime} \in
\mathcal{J}_{j-1}}\sum_{J\in\mathcal{J}(J^\prime)}\sum_{I\in\mathcal{J}_{k-j-1}(J)}\left<|I|^{\frac{1}{2}}\left<U\right>_I^{\frac{1}{2}}M_Jf,|I|^{\frac{1}{2}}\left<U\right>_I^{\frac{1}{2}}M_Jf\right>
\\&=\sum_{J^{\prime} \in
\mathcal{J}_{j-1}}\sum_{J\in\mathcal{J}(J^\prime)}\sum_{I\in\mathcal{J}_{k-j-1}(J)}\left<|I|^{\frac{1}{2}}\left<U\right>_I^{\frac{1}{2}}\langle
U \rangle_J^{-\frac{1}{2}}\langle U
\rangle_J^{\frac{1}{2}}M_Jf,|I|^{\frac{1}{2}}\left<U\right>_I^{\frac{1}{2}}\langle
U \rangle_J^{-\frac{1}{2}}\langle U
\rangle_J^{\frac{1}{2}}M_Jf\right>
\\&=\sum_{J^{\prime} \in
\mathcal{J}_{j-1}}\sum_{J\in\mathcal{J}(J^\prime)}\left<\sum_{I\in\mathcal{J}_{k-j-1}(J)}|I|\left<U\right>_I\langle
U \rangle_J^{-\frac{1}{2}}\langle U
\rangle_J^{\frac{1}{2}}M_Jf,\langle U
\rangle_J^{-\frac{1}{2}}\langle U \rangle_J^{\frac{1}{2}}M_Jf\right>
\\&=\sum_{J^{\prime} \in
\mathcal{J}_{j-1}}\sum_{J\in\mathcal{J}(J^\prime)}\int_{\mathcal{J}_{k-j-1}(J)}||U^{\frac{1}{2}}(x)\langle
U \rangle_J^{-\frac{1}{2}}\langle U
\rangle_J^{\frac{1}{2}}M_Jf||^2dx
\end{align*}

We now apply H\"{o}lder's inequality with $p$ such that $2p$ is the
$r$ from our reverse H\"{o}lder inequality on $U$. Then the above
expression is less than or equal to
\[ \sum_{J^{\prime} \in
\mathcal{J}_{j-1}}\sum_{J\in\mathcal{J}(J^\prime)}\left(\int_{\mathcal{J}_{k-j-1}(J)}||U^{\frac{1}{2}}(x)\langle
U \rangle_J^{-\frac{1}{2}}\langle U
\rangle_J^{\frac{1}{2}}M_Jf||^{2p}dx\right)^{\frac{1}{p}}|\mathcal{J}_{k-j-1}(J)|^\frac{1}{q}.\]

We now use the fact that we are working with a decaying stopping
time to see that this is less than or equal to
\[ \sum_{J^{\prime} \in
\mathcal{J}_{j-1}}\sum_{J\in\mathcal{J}(J^\prime)}\left(\int_{\mathcal{J}_{k-j-1}(J)}||U^{\frac{1}{2}}(x)\langle
U \rangle_J^{-\frac{1}{2}}\langle U
\rangle_J^{\frac{1}{2}}M_Jf||^{2p}dx\right)^{\frac{1}{p}}d^\frac{k-j-1}{q}|J|^{\frac{1}{q}}\]
where $0 < d < 1.$

Now we apply the reverse H\"{o}lder inequality \ref{revh} with
vector $\langle U \rangle_J^{\frac{1}{2}}M_Jf$ to obtain that this
is less than or equal to
\[ \sum_{J^{\prime} \in
\mathcal{J}_{j-1}}\sum_{J\in\mathcal{J}(J^\prime)}||\langle U
\rangle_J^{\frac{1}{2}}M_Jf||^{2}C^{\frac{1}{p}}|J|^{\frac{1}{p}}d^\frac{k-j-1}{q}|J|^{\frac{1}{q}}
= d^{\frac{k-j-1}{q}} C^{\frac{1}{p}}\sum_{J^{\prime} \in
\mathcal{J}_{j-1}}\int_{\cup \mathcal{J}(J^\prime)}
||U(x)^{\frac{1}{2}}M_Jf||^2\]
\[= d^{\frac{k-j-1}{q}} C^{\frac{1}{p}}\sum_{J^{\prime} \in
\mathcal{J}_{j-1}}\int_{\cup \mathcal{J}(J^\prime)} ||S_jf||^2 .\]
This is our core estimate.

To apply Cotlar's Lemma, consider
\[\left<S_k^*S_jf,g\right>_{L^2} =  \left<S_jf,S_kg\right>_{L^2} =
\int_{\cup \mathcal{J}_{k-1}}
\left<S_jf(x),S_kg(x)\right>_{\mathbb{C}^n}dx \leq \int_{\cup
\mathcal{J}_{k-1}}
||S_jf(x)||_{\mathbb{C}^n}||S_kg(x)||_{\mathbb{C}^n}dx\leq\]

\[\left\{\int_{\cup
\mathcal{J}_{k-1}}
||S_jf(x)||^2_{\mathbb{C}^n}\right\}^{\frac{1}{2}}\left\{\int_{\cup
\mathcal{J}_{k-1}}||S_kg(x)||^2_{\mathbb{C}^n}dx\right\}^{\frac{1}{2}}\]

This is true as the support of $S_kf$ is contained in
$\mathcal{J}_{k-1}$ and by Cauchy-Schwartz. We have just dealt with
the relevant bounds for the two factors at the end of this chain of
inequalities.

Also note that \[\left<S_kS_j^*f,g\right>_{L^2} =
\left<S_j^*f,S_k^*g\right>_{L^2}=
\left<\left(S\triangle_j\right)^*f,\left(S\triangle_k\right)^*g\right>_{L^2}=
\left<\triangle_jS^*f,\triangle_kS^*g\right>_{L^2} = 0\] as the
$\triangle_i$ are self adjoint orthogonal projections. This finishes
the proof of Theorem \ref{square}.
\end{proof}

\begin{rem}
The proof of Corollary \ref{maint} also follows from the embedding
theorem of Nazarov and Treil, Theorem \ref{naztreil}. Ideally we
would like to prove this independently of their theorem however we
were unable to do this. $M_V^{-\frac{1}{2}}T_\sigma
M_U^{\frac{1}{2}}$ can be written as
\[M_V^{-\frac{1}{2}}D^{-1}_{V^{-1}} D_{V^{-1}}T_\sigma D_{U}
D^{-1}_{U}M_U^{\frac{1}{2}}.\]  Note that $T_\sigma$ commutes with
$D_{V^{-1}}$ and we can estimate the norm as follows
\[||M_V^{-\frac{1}{2}}T_\sigma M_U^{\frac{1}{2}}|| = ||M_V^{-\frac{1}{2}}D^{-1}_{V^{-1}} D_{V^{-1}}T_\sigma D_{U} D^{-1}_{U}M_U^{\frac{1}{2}}|| \]\[\leq ||M_V^{-\frac{1}{2}} D^{-1}_{V^{-1}}|| ||T_\sigma|| ||D_{V^{-1}}D_{U}|| ||D^{-1}_{U}M_U^{\frac{1}{2}}||.\]
We need conditions on $U$ and $V$ that imply that the operators
$M_V^{-\frac{1}{2}}D^{-1}_{V^{-1}}$, $D_{V^{-1}}D_{U}$ and
$D^{-1}_{U}M_U^{\frac{1}{2}}$ are bounded. Theorem \ref{naztreil}
immediately gives us that $D^{-1}_{U}M_U^{\frac{1}{2}}$ is bounded.
This theorem also applies to $M_V^{-\frac{1}{2}}D^{-1}_{V^{-1}}$ if
we note that its adjoint is $D^{-1}_{V^{-1}}M_V^{-\frac{1}{2}}$. All
we need to show now is that under the hypothesis $D_{V^{-1}}D_{U}$
is a bounded operator. This follows from the joint $A_2$ condition.
\end{rem}

\section{Application to the Hilbert Transform}

As well as showing that the martingale transforms are uniformly
bounded under the conditions of the two main theorems we can also
show that the dyadic shift, $\Sha$ defined below, will be bounded
and hence the Hilbert transform by way of S. Petermichl's averaging
techniques \cite{petermichl},\cite{ptv}.

\begin{defn}\label{shift}
The \emph{dyadic shift} $\Sha$ with respect to the standard dyadic
grid is the operator given by
\[\Sha f = \Sha \sum_{I \in\mathcal{D}}f_Ih_I=\sum_{I\in\mathcal{D}} \langle f,
 h_{I_+} - h_{I_-}\rangle h_I,\] where $f$ is supported on the
 unit interval and has finite Haar expansion.
\end{defn}

\begin{defn}
Define the operator $D_V^+$ as
\[D_V^+:f = \sum_{I\in\mathcal{D}}f_Ih_I \mapsto \sum_{I\in\mathcal{D}} \langle V
\rangle^{\frac{1}{2}}_{I_+} f_I h_I\] and the operator $D_V^-$  as
\[D_V^-:f = \sum_{I\in\mathcal{D}}f_Ih_I \mapsto \sum_{I\in\mathcal{D}} \langle V
\rangle^{\frac{1}{2}}_{I_-} f_I h_I,\] for $f \in L^2(\mathbb{C}^n)$
with finite Haar expansion.

\end{defn}

If we split the shift operator into a sum of two operators, each of
which is bounded,

\[\Sha f = (\Sha_1 + \Sha_2)f = \sum_{I\in \mathcal{D}} \left(\langle f, h_{I_+}\rangle h_I\right) - \sum_{I\in \mathcal{D}} \left(\langle f, h_{I_-}\rangle h_I\right)\]

We can then check that $D_{V^{-1}}^+\Sha_1 D_{V^{-1}}^{-1} = \Sha_1$
and $ D_{V^{-1}}^-\Sha_2 D_{V^{-1}}^{-1} = \Sha_2$ As before we can
estimate $||M_U^{\frac{1}{2}} \Sha M_{V^{-1}}^{\frac{1}{2}}||,$
\[ || M^{\frac{1}{2}}_U (\Sha_1 + \Sha_2) M^{\frac{1}{2}}_{V^{-1}}|| = || M^{\frac{1}{2}}_U \left(D_{V^{-1}}^+\Sha_1 D_{V^{1}}^{-1} + D_{V^{-1}}^-\Sha_2 D_{V^{-1}}^{-1}\right) M^{\frac{1}{2}}_{V^{-1}}||  \]
\[ \leq \left(|| M^{\frac{1}{2}}_U \left(D_{V^{-1}}^+|| ||\Sha_1||\right) + || M^{\frac{1}{2}}_U \left(D_{V^{-1}}^-|| ||\Sha_2||\right)\right)
||D^{-1}_{V^{-1}}M^{\frac{1}{2}}_{V^{-1}}||.
\]
We have already dealt with the boundedness of the third operator and
it is known that $\Sha_1$ and $\Sha_2$ are bounded on unweighted
$L^2$. This leaves the operators $M^{\frac{1}{2}}_U D_{V^{-1}}^+$
and $M^{\frac{1}{2}}_U D_{V^{-1}}^-$.

\[ ||(D_{V^{-1}}^+)M^{\frac{1}{2}}_U f||^2=\langle M^{\frac{1}{2}}_U
(D_{V^{-1}}^+)^2M^{\frac{1}{2}}_U f,f\rangle  \]\[=\langle
(D_{V^{-1}}^+)^2M^{\frac{1}{2}}_U f,M^{\frac{1}{2}}_U
f\rangle=\langle \sum_{I\in\mathcal{D}} \langle V^{-1} \rangle_{I_+}
(U^{\frac{1}{2}}f)_I h_I,M^{\frac{1}{2}}_U f\rangle\]

\[=\sum_{I\in\mathcal{D}}\langle  \langle V^{-1} \rangle_{I_+}
(U^{\frac{1}{2}}f)_I h_I,M^{\frac{1}{2}}_U f\rangle
=\sum_{I\in\mathcal{D}}\frac{1}{|I_+|}\int_{I_+}\langle  V^{-1}(x)
(U^{\frac{1}{2}}f)_I h_I,M^{\frac{1}{2}}_U f\rangle dx\]

\[=\sum_{I\in\mathcal{D}}\frac{1}{|I_+|}\int_{I_+}\langle  V^{-1}(x)
(U^{\frac{1}{2}}f)_I h_I,(U^{\frac{1}{2}}f)_I h_I\rangle dx \leq
\sum_{I\in\mathcal{D}}\frac{1}{|I_+|}\int_{I}\langle  V^{-1}(x)
(U^{\frac{1}{2}}f)_I h_I,M^{\frac{1}{2}}_U f\rangle dx\]

\[= \sum_{I\in\mathcal{D}}\frac{2}{|I|}\int_{I}\langle  V^{-1}(x)
(U^{\frac{1}{2}}f)_I h_I,M^{\frac{1}{2}}_U f\rangle dx = 2\langle
M^{\frac{1}{2}}_U (D_{V^{-1}})^2M^{\frac{1}{2}}_U f,f\rangle.\] The
first inequality is true because we are integrating a positive
function,

$\langle  V^{-1}(x) (U^{\frac{1}{2}}f)_I h_I,(U^{\frac{1}{2}}f)_I
h_I\rangle = \langle  V^{-\frac{1}{2}}(x) (U^{\frac{1}{2}}f)_I
h_I,V^{-\frac{1}{2}}(x)(U^{\frac{1}{2}}f)_I h_I\rangle$, over a
larger interval. The second last equality is due to the fact that
$|I_+| = \frac{1}{2}|I|$. The boundedness of $M^{\frac{1}{2}}_U
D_{V^{-1}}^+$ then follows from our previous bounding of
$M^{\frac{1}{2}}_U D_{V^{-1}}$ and taking adjoints where
appropriate. For $M^{\frac{1}{2}}_U D_{V^{-1}}^-$ the proof is
similar.

\begin{defn}
Instead of this canonical dyadic grid we can define the shift
operator, $\Sha^{\beta,r}$, on the grid $ \mathbb{D}_{r,\beta} =
\left\{r2^m\left([0,1) + l +
\sum_{n<m}2^{i-n}\beta_i\right)\right\}_{l,m \in \mathbb{Z}}$:

\begin{equation*}\label{shift2}\Sha^{\beta,r} f = \Sha^{\beta,r}
\sum_{I
\in\mathbb{D}^{\beta,r}}f_Ih_I=\sum_{I\in\mathbb{D}^{\beta,r}}
\langle f,
 h_{I_+} - h_{I_-}\rangle h_I,\end{equation*}
\end{defn}
The shift operators defined with respect to these dyadic grids  will
be bounded $L^2(V) \rightarrow L^2(U)$ given the joint $A_2$
condition is satisfied, $U$ satisfies the reverse H\"{o}lder
condition and $V^{-1}$ the $A_{2,0}$ condition, all on this new
lattice. The resulting estimate for the norm will be independent of
the lattice.

Assuming the joint $A_2$ condition, that $U$ satisfies the reverse
H\"{o}lder condition and $V$ the $A_{2,0}$ condition, all on
arbitrary intervals, allows us to estimate the norm of the Hilbert
transform in terms of these translated and dilated Haar shifts using
the results from \cite{petermichl} and \cite{ptv}.

\begin{thm}

Let $U$ and $V$ be matrix weights satisfying the joint $A_2$
condition
\[ \left<V^{-1}\right>^{\frac{1}{2}}_I\left<U\right>_I\left<V^{-1}\right>^{\frac{1}{2}}_I < C\] for all  intervals $I$, where $C$ is a constant multiple of the identity.
If $V^{-1} \in A_{2,0}$ and $U$ satisfies the matrix reverse
H\"{o}lder inequality, then the Hilbert Transform is bounded from
$L^2(V)$ to $L^2(U)$.

\end{thm}

\begin{proof}
\[ \left|\left\langle M_U^{\frac{1}{2}} H M_V^{-\frac{1}{2}}  f,g \right\rangle\right|
= C\left|\left\langle M_U^{\frac{1}{2}}
\int_{\{0,1\}^{\mathbb{Z}}}\int_1^2 \Sha^{\beta,r}
M_V^{-\frac{1}{2}}  f \frac{dr}{r}d\mathbb{P}(\beta),g
\right\rangle\right|\]

\[ = C\left|\int_{\{0,1\}^{\mathbb{Z}}}\int_1^2 \left\langle M_U^{\frac{1}{2}} \Sha^{\beta,r}  M_V^{-\frac{1}{2}}  f ,g \right\rangle\frac{dr}{r}d\mathbb{P}(\beta)\right|
\]

\[\leq  C\int_{\{0,1\}^{\mathbb{Z}}}\int_1^2 \left|\left\langle M_U^{\frac{1}{2}} \Sha^{\beta,r}  M_V^{-\frac{1}{2}}  f ,g \right\rangle\right|
\frac{dr}{r}d\mathbb{P}(\beta)\]

\[\leq CC^*\int_{\{0,1\}^{\mathbb{Z}}}\int_1^2 \left|\left\langle f ,g \right\rangle\right|
\frac{dr}{r}d\mathbb{P}(\beta)\leq CC^*||f||||g||,\]

where $C$ is the proportion of the Hilbert Transform to the average
of the shift operators and $C^*$ is the uniform operator norm of the
shift operators.
\end{proof}

The heuristic for adapting our main argument to the case of the
dyadic shift can be applied to a more general class of operators,
\emph{band operators.}

\section{Application to band operators and certain singular integral
operators}

\begin{defn}
A band operator $T$ is a bounded operator on $L^2$ such that there
exists an integer $r > 0$ for which $\langle T h_I, h_J \rangle =0$
for all Haar functions $h_I$, $h_J$ where $J$ is at least a distance
of $r$ away from $I$. By distance we mean tree distance between
dyadic intervals where the tree is formed by connecting each
interval with its parent and children intervals.
\end{defn}
One crucial fact is that, for each $r$ there are only a finite
number of Haar basis elements $h_{\tilde{I}}$ less than tree
distance $r$ from $h_I$. Suppose there are $m$ Haar basis elements
less than $r$ away from each $h_I$ and we label these basis elements
$h_{I_i}$ for $i = 1 .. m$. Then our band operator T will be of the
form

\[f
\mapsto \sum_{I\in\mathcal{D}} \sum_{i=1}^m\phi(I,I_i)\langle f,
h_{I}\rangle h_{I_i},\] where $\phi$ is a function from $\mathcal{D}
\bigoplus \mathcal{D}$ to $\mathbb{C}$.

\begin{lem}
If we have a band operator $T$, written in the form
\[f
\mapsto \sum_{I\in\mathcal{D}} \sum_{i=1}^m\phi(I,I_i)\langle f,
h_{I}\rangle h_{I_i},\]

then the function $\phi:\mathcal{D} \bigoplus \mathcal{D}
\rightarrow \mathbb{C}$ is bounded.
\end{lem}

\begin{proof}
Suppose that $\phi$ is unbounded, as $T$ is a bounded operator we
can choose $I$ and $I_i$ such that $\phi(I,I_i) > ||T||$. Then we
can see that \[||T h_I|| = ||\sum_{I\in\mathcal{D}}
\sum_{i=1}^m\phi(I,I_i)\langle h_I, h_{I}\rangle h_{I_i}|| =
||\sum_{i=1}^m\phi(I,I_i) h_{I_i}|| = \sum_{i=1}^m
|\phi(I,I_i)|||h_{I_i}|| > ||T||,\] condradicting our hypothesis
that $T$ is bounded.
\end{proof}

\begin{figure}[htp]
\includegraphics{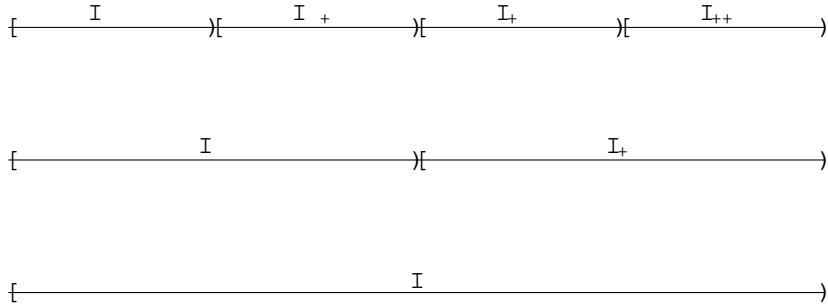}
\caption{A dyadic interval $I$ together with first and second
generation subintervals.}
\end{figure}

\begin{figure}[htp]
\includegraphics{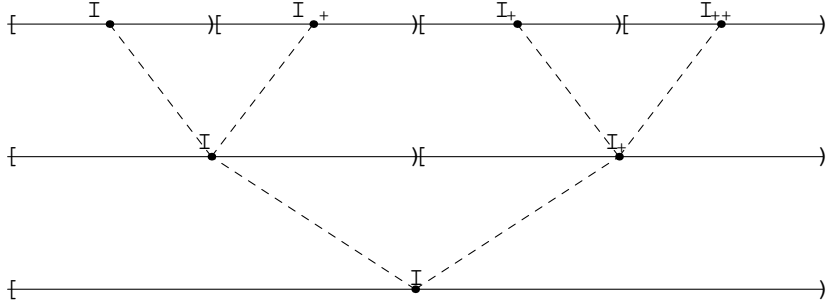}
\caption{The tree formed by connecting dyadic intervals to their
parents and children.}
\end{figure}

\begin{thm}
Let $U$ and $V$ be matrix weights satisfying the dyadic joint $A_2$
condition
\[ \left<V^{-1}\right>^{\frac{1}{2}}_I\left<U\right>_I\left<V^{-1}\right>^{\frac{1}{2}}_I < C\] for all dyadic intervals $I$, where $C$ is a constant multiple of the identity. If $V^{-1} \in A_{2,0}$ and $U$ satisfies the dyadic matrix reverse H\"{o}lder inequality, then
any band operator $T$ is bounded from $L^2(V)$ to $L^2(U)$. If $r$
is the maximum distance associated to the band operator then the
bound will depend only on $r$, the $L^2 \rightarrow L^2$ norm of the
operator and the $A_2$,$A_{2,0}$ and reverse H\"{o}lder constants
associated to the weights.
\end{thm}

\begin{proof}
Again we note that

\[Tf
= \sum_{I\in\mathcal{D}} \sum_{i=1}^m\phi(I,I_i)\langle f,
h_{I}\rangle h_{I_i},\] where $\phi$ is a function from $\mathcal{D}
\bigoplus \mathcal{D}$ to $\mathbb{C}$. $I$ and $I_i$ will always
share an ancestor less than $r$ generations away for each $i=1..m$.
In the case that $I_i$ is a descendant of $I$ then $I$ will be the
common ancestor. In the case where $I_i$ is an ancestor of $I$ then
$I_i$ will be the common ancestor. It is also possible to be in a
situation where neither of these are true but the intervals still
share a common ancestor.

We can split $T$ into a sum of $m$ bounded operators

\[ T = \sum_{i=1}^m T_i,\] where $T_i$ is the operator \[f \mapsto \sum_{I\in\mathcal{D}}\phi(I_i,I)\langle f,
h_{I_i}\rangle h_{I}.\] This sum is constructed so that for each
summand $T_i$ and Haar basis element $h_I$ there is exactly one Haar
coefficient,$\langle f,h_i\rangle$, being mapped to $h_I$. Due to
the nature of the band operator there are at most $m$ Haar
coefficients being mapped to each basis element and thus it is
possible to decompose $T$ into a finite sum of these operators.

We proceed to estimate $||M_U^{\frac{1}{2}} T
M_{V^{-1}}^{\frac{1}{2}}||.$ Note that \[T D_{V^{-1}} = \left(
\sum_{i=1}^m T_i\right) D_{V^{-1}} =  \sum_{i=1}^m D^i_{V^{-1}}T_i
\] where $D^i_{V^{-1}}$ is the operator
\[f \mapsto \sum_{I\in\mathcal{D}} \langle V^{-1}
\rangle^{\frac{1}{2}}_{I_i} f_{I} h_{I}.\]

So

\[||M_U^{\frac{1}{2}} T
M_{V^{-1}}^{\frac{1}{2}}|| = ||M_U^{\frac{1}{2}}
TD_{V^{-1}}D^{-1}_{V^{-1}} M_{V^{-1}}^{\frac{1}{2}}||=
||M_U^{\frac{1}{2}} \left( \sum_{i=1}^m D^i_{V^{-1}} T_i\right)
D^{-1}_{V^{-1}} M_{V^{-1}}^{\frac{1}{2}}|| \]\[\leq \left(
\sum_{i=1}^m||M_U^{\frac{1}{2}}  D^i_{V^{-1}}
T_i||\right)||D^{-1}_{V^{-1}} M_{V^{-1}}^{\frac{1}{2}}||\]
\[\leq \left( \sum_{i=1}^m||M_U^{\frac{1}{2}}  D^i_{V^{-1}}
||||T_i||\right)||D^{-1}_{V^{-1}} M_{V^{-1}}^{\frac{1}{2}}||\] We
know that each $T_i$ is bounded and we have already dealt with the
boundedness of $D^{-1}_{V^{-1}} M_{V^{-1}}^{\frac{1}{2}}$. So it
remains to bound each $M_U^{\frac{1}{2}} D^i_{V^{-1}} .$

So for any $f \in L^2$,

\[||M^{\frac{1}{2}}_U (D^i_{V^{-1}})f||^2 = \langle M^{\frac{1}{2}}_U (D^i_{V^{-1}})^2M^{\frac{1}{2}}_U f, f
\rangle = \langle (D^i_{V^{-1}})^2M^{\frac{1}{2}}_U f,
M^{\frac{1}{2}}_U f \rangle\]
\[= \left\langle\sum_{I\in\mathcal{D}} \langle V^{-1}
\rangle_{I_i} (U^{\frac{1}{2}}f)_{I} h_{I},M^{\frac{1}{2}}_U f
\right\rangle = \sum_{I\in\mathcal{D}}\left\langle \langle V^{-1}
\rangle_{I_i} (U^{\frac{1}{2}}f)_{I} h_{I},M^{\frac{1}{2}}_U f
\right\rangle\]
\[=\sum_{I\in\mathcal{D}}\frac{1}{|I_i|}\int_{I_i}\langle  V^{-1}(x)
(U^{\frac{1}{2}}f)_I h_I,M^{\frac{1}{2}}_U f\rangle dx \leq
\sum_{I\in\mathcal{D}}\frac{2^r}{|I'|}\int_{I'}\langle V^{-1}(x)
(U^{\frac{1}{2}}f)_I h_I,M^{\frac{1}{2}}_U f\rangle dx\] where $I'$
is the common ancestor of $I$ and $I_i$. This is true because each
term \[\langle V^{-1}(x) (U^{\frac{1}{2}}f)_I h_I,M^{\frac{1}{2}}_U
f\rangle = \langle
 V^{-\frac{1}{2}}(x)(U^{\frac{1}{2}}f)_I
h_I,V^{-\frac{1}{2}}(x)(U^{\frac{1}{2}}f)_Ih_I\rangle\] is positive.

We have seen before that if a matrix weight $U$ satisfies the dyadic
$A_{2,0}$ condition then for any vector $\gamma$ the scalar weight
$||U^{\frac{1}{2}}\gamma||^2$ will satisfy the scalar dyadic
$A_\infty$ condition. So if we have a dyadic interval $I$ and a
dyadic interval $J$ contained in $I$ such that the tree distance
between these two is less than $r$, i.e. $|I| \leq 2^r|J|$ then one
of the standard properties of $A_\infty$, see \cite{stein} page 196,
tells us that

\[\beta\int_I||U^{\frac{1}{2}}\gamma||^2 \leq
\int_J||U^{\frac{1}{2}}\gamma||^2\] for some $0<\beta<1$ bounded
away from $0,$ with the bound dependent only on $r$ and the
$A_\infty$ constant.

Using our hypothesis that $V^{-1} \in A_{2,0}$ we can see that

\[\sum_{I\in\mathcal{D}}\frac{2^r}{|I'|}\int_{I'}\langle V^{-1}(x)
(U^{\frac{1}{2}}f)_I h_I,M^{\frac{1}{2}}_U f\rangle dx =
\sum_{I\in\mathcal{D}}\frac{2^r}{|I'|}\int_{I'}\langle
V^{-\frac{1}{2}}(x) (U^{\frac{1}{2}}f)_I h_I,V^{-\frac{1}{2}}(x)
(U^{\frac{1}{2}}f)_I h_I\rangle dx\]
\[\leq
\sum_{I\in\mathcal{D}}\frac{2^r}{|I'|}\int_{I'}||V^{-\frac{1}{2}}(x)
(U^{\frac{1}{2}}f)_I|| ||h_I|| dx \leq
\sum_{I\in\mathcal{D}}\frac{2^r}{|I'|}\int_{I'}||V^{-\frac{1}{2}}(x)
(U^{\frac{1}{2}}f)_I||^2dx \int_{I'}||h_I||^2
dx\]\[=\sum_{I\in\mathcal{D}}\frac{2^r}{|I'|}\int_{I'}||V^{-\frac{1}{2}}(x)
(U^{\frac{1}{2}}f)_I||^2dx \leq
\sum_{I\in\mathcal{D}}\frac{2^r}{\beta|I|}\int_{I}||V^{-\frac{1}{2}}(x)
(U^{\frac{1}{2}}f)_I||^2dx\]
\[=\sum_{I\in\mathcal{D}}\frac{2^r}{\beta}||\left\{\frac{1}{|I|}\int_IV^{-1}(x)dx\right\}^{\frac{1}{2}}
(U^{\frac{1}{2}}f)_I||^2dx = \frac{2^r}{\beta}
||D_{V^{-1}}M_U^{\frac{1}{2}}f||^2.\]

This reduces the estimate of each $D^i_{V^{-1}}M_U^{\frac{1}{2}}$ to
$D_{V^{-1}}M_U^{\frac{1}{2}}$ which was dealt with in Theorem
\ref{square}.
\end{proof}

If $K$ is a function from $\mathbb{R} \setminus \{0\}$ to
$\mathbb{R}$ that is twice differentiable and such that the function
$x^3K(x)$ is almost everywhere bounded and the limit as $x
\rightarrow \infty$ of both $K(x)$ and the first derivative $K'(x)$
are $0$ then the following theorem due to Vagharshakyan's allows us
to apply our hypothesis to singular integral operators of
convolution type with such kernels $K$. Vagharshakyan's theorem
models singular integral operators with such kernels in terms of
translations and dilations of band operators.

\begin{thm}[Vagharshakyan]
If $T$ is a singular integral operator of convolution type with
kernel $K$ as defined above, then $T$ is a positive multiple of the
following operator

\[ f \mapsto \int_{\{0,1\}^{\mathbb{Z}}}\int_1^2 B^{\beta,r}f \frac{dr}{r}d\mathbb{P}(\beta),\]
where $B^{\beta,r}$ is a band operator defined in terms of the
dyadic grid $\mathbb{D}_{\beta,r}$ exactly as they are defined for
the canonical dyadic grid.

\end{thm}

\begin{thm}

Let $U$ and $V$ be matrix weights satisfying the joint $A_2$
condition
\[ \left<V^{-1}\right>^{\frac{1}{2}}_I\left<U\right>_I\left<V^{-1}\right>^{\frac{1}{2}}_I < C\] for all  intervals $I$, where $C$ is a constant multiple of the identity.
If $V^{-1} \in A_{2,0}$ and $U$ satisfies the matrix reverse
H\"{o}lder inequality, then the singular integral operator of
convolution type with kernel $K$ is bounded from $L^2(V)$ to
$L^2(U)$.

\end{thm}

\begin{proof}
\[ \left|\left\langle M_U^{\frac{1}{2}} T M_V^{-\frac{1}{2}}  f,g \right\rangle\right|
= \tilde{C}\left|\left\langle M_U^{\frac{1}{2}}
\int_{\{0,1\}^{\mathbb{Z}}}\int_1^2 B^{\beta,r}  M_V^{-\frac{1}{2}}
f \frac{dr}{r}d\mathbb{P}(\beta),g \right\rangle\right|\]

\[ = \tilde{C}\left|\int_{\{0,1\}^{\mathbb{Z}}}\int_1^2 \left\langle M_U^{\frac{1}{2}} B^{\beta,r}  M_V^{-\frac{1}{2}}  f ,g \right\rangle\frac{dr}{r}d\mathbb{P}(\beta)\right|
\]

\[\leq  \tilde{C}\int_{\{0,1\}^{\mathbb{Z}}}\int_1^2 \left|\left\langle M_U^{\frac{1}{2}} B^{\beta,r}  M_V^{-\frac{1}{2}}  f ,g \right\rangle\right|
\frac{dr}{r}d\mathbb{P}(\beta)\]

\[\leq \tilde{C}C^*\int_{\{0,1\}^{\mathbb{Z}}}\int_1^2 \left|\left\langle f ,g \right\rangle\right|
\frac{dr}{r}d\mathbb{P}(\beta)\leq \tilde{C}C^*||f||||g||,\]

where $\tilde{C}$ is the constant multiple of the singular integral
operator corresponding to the average of the band operators and
$C^*$ is the operator norm of the band operators. Note by uniform
norm we mean that a particular band operator then defined with
respect to different dyadic grids will have the same operator norm.
\end{proof}

\bibliographystyle{amsplain}

\end{document}